\documentclass[preprint,12pt]{elsarticle}
\usepackage{amsthm,amsmath,amssymb}
\usepackage{graphicx}
\usepackage{longtable,booktabs}
\usepackage[colorlinks=true,citecolor=black,linkcolor=black,urlcolor=blue]{hyperref}
\usepackage{mathrsfs}
\usepackage{epstopdf}
\usepackage{epsfig}
\allowdisplaybreaks[4]

\theoremstyle{plain}
\newtheorem{thm}{Theorem}
\newtheorem{lem}[thm]{Lemma}
\newtheorem{cor}[thm]{Corollary}

\theoremstyle{definition}

\newtheorem{example}[thm]{Example}

\theoremstyle{remark}
\newtheorem{rem}[thm]{Remark}

\journal{Graphs and Combinatorics}
\begin{document}
\title{The Ihara-zeta function and the spectrum of the join of two semi-regular bipartite graphs}
\author{Xiaotong Li\\
\small School of Mathematical Sciences\\[-0.8ex]
\small Xiamen University\\[-0.8ex]
\small P. R. China\\
Xian'an Jin\footnote{Corresponding author.}\\
\small School of Mathematical Sciences\\[-0.8ex]
\small Xiamen University\\[-0.8ex]
\small P. R. China\\
Qi Yan\\
\small School of Mathematics\\[-0.8ex]
\small China University of Mining and Technology\\[-0.8ex]
\small P. R. China\\
\small{\tt Email: xiaotongli@stu.xmu.edu.cn; xajin@xmu.edu.cn; qiyan@cumt.edu.cn}
}

\begin{abstract}
In this paper, using matrix techniques, we compute the Ihara-zeta function and the number of spanning trees of the join of two semi-regular bipartite graphs. Furthermore, we show that the spectrum and the zeta function of the join of two semi-regular bipartite graphs can determine each other.
\end{abstract}


\maketitle

\section{Introduction}
The {\it Ihara-zeta function} \cite{I} of a finite connected graph (without degree-1 vertices) $G$ was originally defined to be the following function of the complex number $u$ with the modulus $|u|$ sufficiently small:
\begin{eqnarray*}
Z_G(u)=\prod_{[C]}(1-u^{|C|})^{-1},
\end{eqnarray*}
where the product is over all equivalent classes $[C]$ of primitive, backtrack-less, tailless closed paths $C$ and $|C|$ is the length of $C$. We will refer to Ihara-zeta functions of graphs as zeta functions of graphs when there is no ambiguity. We refer the reader to \cite{AT} for acquiring in-depth knowledge of zeta functions of graphs.

Suppose the vertex set of a graph $G$ is $\{v_1,v_2,\cdots,v_n\}$, and the corresponding degrees of $v_1,v_2,\cdots,v_n$ are $d_1,d_2,\cdots,d_n$. Then the diagonal matrix $D(G):=diag (d_1,d_2,\cdots,d_n)$ is called the {\it degree matrix} of $G$. The {\it adjacency matrix} of $G$ is the $n\times n$ matrix $A(G):=(a_{ij})$, where $a_{ij}$ is the number of edges joining vertices $v_i$ and $v_j$, each loop counting as two edges. Let $Q(G):=D(G)-I_n$, where $I_n$ is the identity matrix of order $n$. Let
\begin{eqnarray*}
f_G(u)=\det(I_n-uA(G)+u^2Q(G)).
\end{eqnarray*}
Bass \cite{B} proved that the reciprocal of the zeta function can be computed from the determinant $f_G(u)$:
\begin{eqnarray*}
Z_G(u)^{-1}=(1-u^2)^{m-n}f_G(u),
\end{eqnarray*}
where $m$ is the number of edges of $G$.

The zeta function of a graph $G$ contains much information of the graph such as the number of edges and the number of loops of the graph \cite{CD}, the number of vertices of the graph \cite{CO}. Let $\tau(G)$ be the number of spanning trees (i.e. the complexity) of $G$. Then Northshield \cite{N} proved:
\begin{eqnarray*}
f'_G(1)=2(m-n)\tau(G).
\end{eqnarray*}

By now, zeta functions of many graph families have been studied. Sato \cite {SA} obtained formulas of the zeta function and the complexity of the line graph of a semi-regular bipartite graph. He \cite {SAT} further obtained the zeta function and the complexity of the middle graph of a semi-regular bipartite graph. Bayati and Somodi \cite {BP} derived a formula of the zeta function of the cone over a regular graph, and showed that the zeta function and the adjacency spectrum of the cone over a regular graph can determine each other. Blanchard et al. \cite {BA} studied the zeta function of the $r$-join of two regular graphs, and proved that their zeta functions are the same if and only if they have the same adjacency spectrum. Recently, Chen and Chen \cite {C} studied Bartholdi zeta functions of the generalized join of some regular graphs along a fixed graph, Li and Hou \cite {LDQ} studied the zeta function of the cone of a semi-regular bipartite graph and Li et al. \cite {LLH} studied zeta functions of three types of corona graphs.

Motivated by the above works, in this article, we study the zeta function and the number of spanning trees of the join of two semi-regular bipartite graphs by means of the matrix techniques. Furthermore, we show that the spectrum and the zeta function of the join of two semi-regular bipartite graphs can determine each other.

\section{Preliminaries}\label{s2}
If the vertex set $V(G)$ of a graph $G$ can be partitioned into two non-empty subsets $V_1(G)$ and $V_2(G)$, such that each edge of $G$ has one end in $V_1(G)$ and the other end in $V_2(G)$, then $G$ is called a {\it bipartite graph}. In particular, if the degree of each vertex of $V_1(G)$ and $V_2(G)$ are $q_1$ and $q_2$, respectively, then $G$ is called a $(q_1,q_2)$-{\it semi-regular bipartite graph}.

The {\it join} $G_1\vee G_2$ of two graphs $G_1$ and $G_2$ is defined to be the graph with $V(G_1\vee G_2)=V(G_1)\cup V(G_2)$ and $E(G_1\vee G_2)=E(G_1)\cup E(G_2)\cup \{uv|u\in V(G_1),v\in V(G_2)\}$.

The polynomial $\Phi_A(G, x)=\det(xI_n-A(G))$ is said to be the {\it characteristic polynomial} of the graph $G$. The set of all zeros of the characteristic polynomial of the graph $G$ (i.e. all eigenvalues of its adjacency matrix) is said to be the {\it spectrum} of the graph $G$, denoted by $spec(G)$. Similarly, $\Phi_L(G, x)=\det(xI_n-L(G))$ is said to be the {\it Laplacian characteristic polynomial} of the graph $G$, where $L(G)=D(G)-A(G)$ is the {\it Laplacian matrix} of the graph $G$, and the set of all zeros of the Laplacian characteristic polynomial of the graph $G$ (i.e. all eigenvalues of its Laplacian matrix) is said to be the {\it Laplacian spectrum} of the graph $G$.

The following notations will be used throughout this article.

$G_1$: $(q_1,q_2)$-semi-regular bipartite graph with  $|V_1|=n_1$ and $|V_2|=n_2$ $(n_1\leq n_2)$. Let $\nu_1=n_1+n_2$ and $\epsilon_1=n_1q_1=n_2q_2$.

$G_2$: $(q_3,q_4)$-semi-regular bipartite graph with $|V_3|=n_3$ and $|V_4|=n_4$ $(n_3\leq n_4)$. Let $\nu_2=n_3+n_4$ and $\epsilon_2=n_3q_3=n_4q_4$.

$G_{2}^{'}$: $(q_{3}',q_{4}')$-semi-regular bipartite graph with $|V_{3}'|=n_{3}'$ and $|V_{4}'|=n_{4}'$ $(n_{3}'\leq n_{4}')$. Let $ \nu_{2}'=n_{3}'+n_{4}'$ and $\epsilon_{2}'=n_{3}'q_{3}'=n_{4}'q_{4}'$.

$J_{p\times q}$: $p\times q$ matrix with all entries equal to 1.

$1_n$: a column vector of size $n$ with all entries equal to 1.

$A^{T}$: the transpose of the matrix $A$.

$A^\ast$: the adjugate matrix of $A$.

The following three lemmas \cite{LLH, LZ, Z} on matrix techniques will be used. Let $M_{11},M_{12},M_{21}$ and $M_{22}$ be $n_1\times n_1, n_1\times n_2, n_2\times n_1$ and $n_2\times n_2$ matrices, respectively. Let
$$M=\left[
\begin{array}{ll}
M_{11}&M_{12}\\
M_{21}&M_{22}\
\end{array}
\right].$$

\begin{lem}\label{lem3.1.3}
\[\det M=\left\{\begin{array}{ll}
\det(M_{11})\times \det(M_{22}-M_{21}M_{11}^{-1}M_{12}),& \text{if $M_{11}$ is invertible};\\
\det(M_{22})\times \det(M_{11}-M_{12}M_{22}^{-1}M_{21}),& \text{if $M_{22}$ is invertible}.
\end{array}\right.\]
\end{lem}

\begin{lem}\label{lem3.1.1}
Let $n=n_1+n_2$. If $M$ is invertible, the row sums of $M_{11}, M_{12}, M_{21}$ and $M_{22}$ are $r_1, r_2, r_3$ and $r_4$, respectively and $r_{1}r_{4}-r_{2}r_{3}\neq 0$, then
\begin{eqnarray*}
1_n^TM^{-1}1_n=\frac{n_1(r_4-r_2)+n_2(r_1-r_3)}{r_1r_4-r_2r_3}.
\end{eqnarray*}
\end{lem}

\begin{lem}\label{lem3.1.2}
Let $A$ be an $n\times n$ real matrix and $\alpha$ be a real number. Then
\begin{eqnarray*}
\det(A+\alpha J_{n\times n})\!=\!\det A+\alpha1_n^TA^\ast1_n.
\end{eqnarray*}
\end{lem}

In \cite{H}, Hashimoto obtained the spectrum of semi-regular bipartite graphs as follows.
\begin{lem}\label{lem3.1.4}
Let $G_1$ be a $(q_1,q_2)$-semi-regular bipartite graph.
Then
\begin{eqnarray*}
spec(G_1)=\{ \pm \lambda _1,\pm\lambda _2,\cdot \cdot \cdot ,\pm\lambda _{n_1},\overbrace{0,\cdot \cdot \cdot ,0}^{n_2-n_1}\},
\end{eqnarray*}
where $\sqrt{q_1q_2}=\lambda _1\geq\lambda _2\geq\cdot \cdot \cdot \geq\lambda _{n_1}\geq0$.
\end{lem}
\begin{rem}
In Lemma \ref{lem3.1.4}, $\lambda_{n_1}$ may be 0. For example, the spectrum of the complete bipartite graph $K_{2,3}$ is $\{\pm\sqrt6,0,0,0\}$.
\end{rem}

We assume that the number of positive eigenvalues of adjacency matrices of $G_1$ and $G_2$ are $k_1$ and $k_2$, respectively, so the spectrum of $G_1$ and $G_2$ can be written as:
\begin{eqnarray*}
spec(G_1)=\{ \pm \lambda _1,\pm\lambda _2,\cdot \cdot \cdot ,\pm\lambda _{k_1},\overbrace{0,\cdot \cdot \cdot ,0}^ {\nu_1-2k_1}\},\\
spec(G_2)=\{ \pm \mu_1,\pm\mu _2,\cdot \cdot \cdot ,\pm\mu _{k_2},\overbrace{0,\cdot \cdot \cdot ,0}^{ \nu_2-2k_2}\},
\end{eqnarray*}
where $k_1\leq n_1$ and $k_2\leq n_3$.

\section{The spectrum and the zeta function}\label{s3}

In this section we study the spectrum and the zeta function of the join of two semi-regular bipartite graphs.

\begin{thm}\rm\label{2.7}
Let $G=G_1\vee G_2$.  Let $\theta_1,\theta_2,\theta_3$ and $\theta_4$ be zeros of
$f(\lambda)=\lambda^4-(q_1q_2+q_3q_4+\nu_1\nu_2)\lambda^2-2(\nu_1\epsilon_2+\nu_2\epsilon_1)\lambda+q_1q_2q_3q_4-4\epsilon_1\epsilon_2.$
Then
$$
spec(G)=\{\theta_1,\theta_2,\theta_3,\theta_4,\pm\lambda_2,\cdots,\pm\lambda_{k_{1}},\pm\mu_2,\cdots,\pm\mu_{k_{2}},\overbrace{0,\cdot\cdot\cdot\cdot\cdot\cdot,0}^{\nu_1-2k_1+\nu_2-2k_2}\}.
$$
\end{thm}
\begin{proof}
Note that the adjacency matrix $A(G)$ of $G$ can be written as:
\begin{eqnarray*}
\left[
\begin{array}{cccc}
0   &  E    & J_{n_1\times n_3} & J_{n_1\times n_4}\\
E^T &  0    & J_{n_2\times n_3} & J_{n_2\times n_4}\\
J_{n_3\times n_1} &J_{n_3\times n_2}  &0     &F\\
J_{n_4\times n_1} &J_{n_4\times n_2}  &F^T   &0\\
\end{array}
\right].
\end{eqnarray*}
By Lemma \ref{lem3.1.3}, the characteristic polynomial $\Phi_A(G,\lambda)$ of $G$ is

\begin{eqnarray*}
 \Phi_A(G,\lambda)\!\!\!&=&\!\!\!\det\left[
                  \begin{array}{cccc}
\lambda I_{n_1}   &  -E                 &      -J_{n_1\times n_3}&-J_{n_1\times n_4}\\
-E^T              & \lambda I_{n_2}     & -J_{n_2\times n_3}     &-J_{n_2\times n_4}\\
-J_{n_3\times n_1}&-J_{n_3\times n_2}   &\lambda I_{n_3}         &-F\\
-J_{n_4\times n_1}&-J_{n_4\times n_2}   &-F^T                    &\lambda I_{n_4}\\
                  \end{array}
                  \right]\\
                  \!\!\!&=&\!\!\!
                  \det \left[\left[
                  \begin{array}{cc}
                  \lambda I_{n_1}&  -E\\
                  -E^T                         & \lambda I_{n_2}\\
                  \end{array}
                  \right]-J_{\nu_1\times \nu_2}\left[
                  \begin{array}{cc}
                 \lambda I_{n_3}      &-F\\
                  -F^T                               &\lambda I_{n_4}\\
                  \end{array}
             \right] ^{-1}J_{\nu_2\times \nu_1}\right]\\
             & &\!\!\!\det \left[
                  \begin{array}{cc}
                  \lambda I_{n_3}      &-F\\
                  -F^T                               &\lambda I_{n_4}\\
                   \end{array}
                  \right].\label{eq1}
\end{eqnarray*}
Note that the sum of entries of each row (resp. column) of $E$ is $q_1$ (resp. $q_2)$ and the sum of entries of each row (resp. column) of $F$ is $q_3$ (resp. $q_4)$. By Lemma \ref{lem3.1.1} with $r_1=r_4=\lambda, r_2=-q_3, r_3=-q_4$, we obtain
\begin{eqnarray*}
1_{\nu_2}^T\left[
                  \begin{array}{cc}
                 \lambda I_{n_3}&-F\\
                  -F^T& \lambda I_{n_4}\\
                  \end{array}
             \right] ^{-1}1_{\nu_2}&=&\frac {n_3(\lambda+q_3)+n_4(\lambda+q_4)}{\lambda^2-q_3q_4}.
\end{eqnarray*}
Hence,
\begin{eqnarray*}
J_{\nu_1\times \nu_2}\left[
                  \begin{array}{cc}
                  \lambda I_{n_3}&-F\\
                  -F^T& \lambda I_{n_4}\\
                  \end{array}
             \right] ^{-1}J_{\nu_2\times \nu_1}&=&c_{1}J_{\nu_{1}\times \nu_1},
\end{eqnarray*}
where
$c_1=\frac {n_3(\lambda+q_3)+n_4(\lambda+q_4)}{\lambda^2-q_3q_4}$.
By Lemma \ref{lem3.1.2},
\begin{eqnarray*}
& &\det\left[\left[
                  \begin{array}{cc}
                  \lambda I_{n_1}&-E\\
-E^T& \lambda I_{n_2}\\
                  \end{array}
             \right]-c_1J_{\nu_1\times \nu_1}\right]\nonumber\\
             &=&\det\left[
                  \begin{array}{cc}
                  \lambda I_{n_1}&-E\\
-E^T& \lambda I_{n_2}\\
                  \end{array}
             \right]-c_1 1_{\nu_1}^{T}\left[
                  \begin{array}{cc}
                  \lambda I_{n_1}&-E\\
-E^T& \lambda I_{n_2}\\
                  \end{array}
             \right]^\ast1_{\nu_1}\\
             &=&
             \det\left[
                  \begin{array}{cc}
                  \lambda I_{n_1}&-E\\
-E^T& \lambda I_{n_2}\\
                  \end{array}
             \right]\left(1-c_11_{\nu_1}^{T}\left[
                  \begin{array}{cc}
                  \lambda I_{n_1}&-E\\
-E^T& \lambda I_{n_2}\\
                  \end{array}
             \right]^{-1}1_{\nu_1}\right)\\
             &=&
             \det\left[
                  \begin{array}{cc}
                  \lambda I_{n_1}&-E\\
-E^T& \lambda I_{n_2}\\
                  \end{array}
             \right]\left(1-c_1c_2\right),\label{eq2}
\end{eqnarray*}
where $c_2=\frac {n_1(\lambda+q_1)+n_2(\lambda+q_2)}{\lambda^2-q_1q_2}.$ Then
\begin{eqnarray*}
\Phi_A(G,\lambda)\!\!\!&=&\!\!\!\det  \left[
                  \begin{array}{cc}
                  \lambda I_{n_3}&-F\\
                  -F^T& \lambda I_{n_4}\\
                  \end{array}
                 \right]\det\left[
                  \begin{array}{cc}
                  \lambda I_{n_1}&-E\\
-E^T& \lambda I_{n_2}\\
                  \end{array}
             \right](1-c_1c_2).\label{eq1}
\end{eqnarray*}
Let $spec(G_1^2)$ and $spec(G_2^2)$ be the spectra of $A^2(G_1)=\left[
\begin{array}{ll}
EE^T&0\\
0&E^TE\\
\end{array}
\right]$ and $A^2(G_2)=\left[
\begin{array}{ll}
FF^T&0\\
0&F^TF\\
\end{array}
\right]$, respectively. Then
\begin{eqnarray*}
spec(G_1^2)=\{\lambda _1^2,\lambda _2^2,\cdot \cdot \cdot ,\lambda _{k_1}^2,\lambda _1^2,\lambda _2^2,\cdot \cdot \cdot ,\lambda _{k_1}^2,\overbrace{0,\cdot \cdot \cdot ,0}^{\nu_1-2k_1}\},\\
spec(G_2^2)=\{\mu_1^2,\mu _2^2,\cdot \cdot \cdot ,\mu _{k_2}^2,\mu_1^2,\mu _2^2,\cdot \cdot \cdot ,\mu _{k_2}^2,\overbrace{0,\cdot \cdot \cdot ,0}^{\nu_2-2k_2}\}.
\end{eqnarray*}
Since $EE^T$ and $E^TE$ have the same non-zero eigenvalues and $FF^T$ and $F^TF$ have the same non-zero eigenvalues. Thus,
\begin{eqnarray*}
spec(EE^T)=\{\lambda _1^2,\lambda _2^2,\cdot \cdot \cdot ,\lambda _{k_1}^2,\overbrace{0,\cdot \cdot \cdot ,0}^{n_1-k_1}\},\label{eq4}\\
spec(FF^T)=\{\mu_1^2,\mu _2^2,\cdot \cdot \cdot ,\mu _{k_2}^2,\overbrace{0,\cdot \cdot \cdot ,0}^{n_3-k_2}\}.\label{eq5}
\end{eqnarray*}
By Lemma \ref{lem3.1.3},
\begin{eqnarray*}
\det\left[
\begin{array}{cc}
                  \lambda I_{n_3}&-F\\
                  -F^T& \lambda I_{n_4}\\
                  \end{array}
                  \right]
                  &=&\lambda ^{n_4}\det(\lambda I_{n_3}-\frac {1}{\lambda}FF^T)\\
                  &=&\lambda ^{n_4-n_3}\det(\lambda^2I_{n_3}-FF^T)\\
                  &=&\lambda ^{n_4-n_3}(\lambda ^2)^{n_3-k_2}\prod_{j=1}^{k_2}\left(\lambda^2-\mu_{j}^{2}\right)\\
                   &=&\lambda^{\nu_3-2k_2}\prod_{j=1}^{k_2}\left(\lambda^2-\mu_{j}^{2}\right)
\end{eqnarray*}
and
\begin{eqnarray*}
                  \det\left[
                  \begin{array}{cc}
                  \lambda I_{n_1}&-E\\
                 -E^T& \lambda I_{n_2}\\
                  \end{array}
             \right]
             &=&\lambda^{n_2}\det(\lambda I_{n_1}-\frac{1}{\lambda}EE^T)\\
             &=&\lambda^{n_2-n_1}\det(\lambda^2I_{n_1}-EE^T)\\
             &=&\lambda^{n_2-n_1}(\lambda ^2)^{n_1-k_1}\prod_{i=1}^{k_1}\left(\lambda^2-\lambda_{i}^{2}\right)\\
             &=&\lambda^{\nu_1-2k_1}\prod_{i=1}^{k_1}\left(\lambda^2-\lambda_{i}^{2}\right).
\end{eqnarray*}
Thus,
\begin{eqnarray*}
\Phi_A(G,\lambda)\!\!\!&=&\!\!\!\lambda^{\nu_1-2k_1+\nu_2-2k_2}(1-c_1c_2)\prod_{i=1}^{k_1}\left(\lambda^2-\lambda_{i}^{2}\right)\prod_{j=1}^{k_2}\left(\lambda^2-\mu_{j}^{2}\right)\\
                 \!\!\!&=&\!\!\!\lambda^{\nu_1-2k_1+\nu_2-2k_2}f(\lambda)\prod_{i=2}^{k_1}\left(\lambda^2-\lambda_{i}^{2}\right)\prod_{j=2}^{k_2}\left(\lambda^2-\mu_{j}^{2}\right).
\end{eqnarray*}
\end{proof}

\begin{rem}\label{re1}\rm
In general, 4 zeros of $f(\lambda)$ are all distinct. Furthermore, any two are not symmetric with respect to 0. If not, suppose that $\theta_1+\theta_2=0$. Since $f(\theta_1)=\theta_1^4-(q_1q_2+q_3q_4+\nu_1\nu_2)\theta_1^2-2(\nu_1\epsilon_2+\nu_2\epsilon_1){\theta_1}+q_1q_2q_3q_4-4\epsilon_1\epsilon_2=0$, it follows that $f(\theta_2)=f(-\theta_1)=\theta_1^4-(q_1q_2+q_3q_4+\nu_1\nu_2)\theta_1^2+2(\nu_1\epsilon_2+\nu_2\epsilon_1){\theta_1}+q_1q_2q_3q_4-4\epsilon_1\epsilon_2=4(\nu_1\epsilon_2+\nu_2\epsilon_1){\theta_1}\neq0$, a contradiction.
\end{rem}

The same approach is used for us to obtain the zeta function of the join of two semi-regular bipartite graphs.

\begin{thm}\label{thm3.2.1} Let $G=G_1\vee G_2$. Then
\begin{eqnarray*}
Z_G(u)^{-1}&=&(1-u^2)^{\epsilon_1+\epsilon_2+\nu_1\nu_2-\nu_1-\nu_2}x_1^{n_1-k_1}x_2^{n_2-k_1}x_3^{n_3-k_2}x_4^{n_4-k_2}h(u)\\
           & &\prod_{i=2}^{k_1}(x_1x_2-\lambda _i^2u^2)\prod_{j=2}^{k_2}(x_3x_4-\mu _j^2u^2),
\end{eqnarray*}
where $x_1=1+(q_1+\nu_2-1)u^2, x_2=1+(q_2+\nu_2-1)u^2, x_3=1+(q_3+\nu_1-1)u^2, x_4=1+(q_4+\nu_1-1)u^2,$ and
$$h(u)=(1-a_1a_2u^2 )(x_1x_2-q_1q_2u^2)(x_3x_4-q_3q_4u^2)$$ with $a_1=\frac {n_3(x_4+uq_3)+n_4(x_3+uq_4)}{x_3x_4-q_3q_4u^2}$ and
$a_2=\frac {n_1(x_2+uq_1)+n_2(x_1+uq_2)}{x_1x_2-q_1q_2u^2}.$
\end{thm}
\begin{proof} Recall that
$$
A(G)= \left[
\begin{array}{ll}
A(G_{1})&J_{\nu_1\times \nu_2}\\
J_{\nu_2\times \nu_1}&A(G_{2})\\
\end{array}
\right],
$$
where
$
A(G_1)= \left[
\begin{array}{ll}
0&E\\
E^T&0\\
\end{array}
\right]$ and
$A(G_2)=\left[
\begin{array}{ll}
0&F\\
F^T&0\\
\end{array}
\right].
$

$$
D(G)= \left[
\begin{array}{cccc}
(q_1+\nu_2)I_{n_1}& & & \\
 &(q_2+\nu_2)I_{n_2}& & \\
 & &(q_3+\nu_1)I_{n_3}& \\
 & & &(q_4+\nu_1)I_{n_4}\\
\end{array}
\right].
$$
Thus
\begin{eqnarray*}
f_G(u) \!\!\!&=&\!\!\!\det\left(I_{\nu_1+\nu_2}-uA(G)+u^2(D(G)-I_{\nu_1+\nu_2})\right) \\
       \!\!\!&=&\!\!\!\det \left[
                  \begin{array}{cccc}
                 x_1I_{n_1}&-uE&-uJ_{n_1\times n_3}&-uJ_{n_1\times n_4}\\
-uE^T&x_2I_{n_2}&-uJ_{n_2\times n_3}&-uJ_{n_2\times n_4}\\
-uJ_{n_3\times n_1}&-uJ_{n_3\times n_2}&x_3I_{n_3}&-uF\\
-uJ_{n_4\times n_1}&-uJ_{n_4\times n_2}&-uF^T&x_4I_{n_4}\\
                  \end{array}
                   \right].
\end{eqnarray*}
By Lemma \ref{lem3.1.1},
\begin{eqnarray*}
1_{\nu_2}^T\left[
                  \begin{array}{cc}
                  x_3I_{n_3}&-uF\\
                  -uF^T& x_4I_{n_4}\\
                  \end{array}
             \right] ^{-1}1_{\nu_2}&=&a_1.
\end{eqnarray*}
Hence
\begin{eqnarray*}
J_{\nu_1\times \nu_2}\left[
                  \begin{array}{cc}
                  x_3I_{n_3}&-uF\\
                  -uF^T& x_4I_{n_4}\\
                  \end{array}
             \right] ^{-1}J_{\nu_2\times \nu_1}&=&a_1J_{\nu_1\times \nu_1}.
\end{eqnarray*}
By Lemmas \ref{lem3.1.3}, \ref{lem3.1.1} and \ref{lem3.1.2},
\begin{eqnarray*}
  f_G(u)\!\!\!&=&\!\!\!\det \left[\left[
                  \begin{array}{cc}
                  x_1I_{n_1}&-uE\\
                  -uE^T& x_2I_{n_2}\\
                  \end{array}
                  \right]-u^2J_{\nu_1\times \nu_2}\left[
                  \begin{array}{cc}
                  x_3I_{n_3}&-uF\\
                  -uF^T& x_4I_{n_4}\\
                  \end{array}
             \right] ^{-1}J_{\nu_2\times \nu_1}\right]\\
             & &\!\!\!\det  \left[
                  \begin{array}{cc}
                  x_3I_{n_3}&-uF\\
                  -uF^T& x_4I_{n_4}\\
                  \end{array}
                  \right]\\
             &=&\!\!\!\det\left[\left[
                  \begin{array}{cc}
                  x_1I_{n_1}&-uE\\
                  -uE^T& x_2I_{n_2}\\
                  \end{array}
             \right]-a_1u^2J_{\nu_1\times \nu_1}\right]
             \det  \left[
                  \begin{array}{cc}
                  x_3I_{n_3}&-uF\\
                  -uF^T& x_4I_{n_4}\\
                  \end{array}
             \right]\\
       &=&\!\!\!
      \left(\det\left[
                  \begin{array}{cc}
                  x_1I_{n_1}&-uE\\
-uE^T& x_2I_{n_2}\\
                  \end{array}
             \right]-a_1u^2 1_{\nu_1}^{T}\left(
                  \begin{array}{cc}
                  x_1I_{n_1}&-uE\\
-uE^T& x_2I_{n_2}\\
                  \end{array}
             \right)^\ast1_{\nu_1}\right)\\
           & &\!\!\!\det\left[
                  \begin{array}{cc}
                  x_3I_{n_3}&-uF\\
                  -uF^T& x_4I_{n_4}\\
                  \end{array}
                 \right]\\
&=&\!\!\!
                 \det\left[
                  \begin{array}{cc}
                  x_1I_{n_1}&-uE\\
-uE^T& x_2I_{n_2}\\
                  \end{array}
             \right]
             \left(1-a_1u^2 1_{\nu_1}^{T}\left(
                  \begin{array}{cc}
                  x_1I_{n_1}&-uE\\
-uE^T& x_2I_{n_2}\\
                  \end{array}
             \right)^{-1}1_{\nu_1}\right)\\
             & &\!\!\!\det\left[
                  \begin{array}{cc}
                  x_3I_{n_3}&-uF\\
                  -uF^T& x_4I_{n_4}\\
                  \end{array}
                 \right]\\
             &=&\!\!\!\det\left[
                  \begin{array}{cc}
                  x_3I_{n_3}&-uF\\
                  -uF^T& x_4I_{n_4}\\
                  \end{array}
                 \right]
                 \det\left[
                  \begin{array}{cc}
                  x_1I_{n_1}&-uE\\
-uE^T& x_2I_{n_2}\\
                  \end{array}
             \right]
             \left(1-a_1a_2u^2\right),
\end{eqnarray*}
and
\begin{eqnarray*}
\det  \left[
                  \begin{array}{cc}
                  x_3I_{n_3}&-uF\\
                  -uF^T& x_4I_{n_4}\\
                  \end{array}
                  \right]
                  &=&x_4^{n_4}\det(x_3I_{n_3}-\frac {u^2}{x_4}FF^T)\\
                  &=&x_4^{n_4-n_3}\det(x_3x_4I_{n_3}-u^2FF^T),\\
                  \det\left[
                  \begin{array}{cc}
                  x_1I_{n_1}&-uE\\
                 -uE^T& x_2I_{n_2}\\
                  \end{array}
             \right]
             &=&x_2^{n_2}\det(x_1I_{n_1}-\frac{u^2}{x_2}EE^T)\\
             &=&x_2^{n_2-n_1}\det(x_1x_2I_{n_1}-u^2EE^T).
\end{eqnarray*}
Then
\begin{eqnarray*}
f_G(u) &=&(1-a_1a_2u^2)x_2^{n_2-n_1}\det(x_1x_2I_{n_1}-u^2EE^T)\\
       & &x_4^{n_4-n_3}\det(x_3x_4I_{n_3}-u^2FF^T).
\end{eqnarray*}
In the proof of Theorem \ref{2.7}, we obtained that
\begin{eqnarray*}
spec(EE^T)=\{\lambda _1^2,\lambda _2^2,\cdot \cdot \cdot ,\lambda _{k_1}^2,\overbrace{0,\cdot \cdot \cdot ,0}^{n_1-k_1}\},\\
spec(FF^T)=\{\mu_1^2,\mu _2^2,\cdot \cdot \cdot ,\mu _{k_2}^2,\overbrace{0,\cdot \cdot \cdot ,0}^{n_3-k_2}\}.
\end{eqnarray*}
Hence,
\begin{eqnarray*}
f_G(u) &=&(1-a_1a_2u^2 )x_2^{n_2-n_1}x_4^{n_4-n_3}(x_1x_2)^{n_1-k_1}(x_3x_4)^{n_3-k_2}\\
       & &\prod_{i=1}^{k_1}(x_1x_2-\lambda _i^2u^2) \prod_{j=1}^{k_2}(x_3x_4-\mu _j^2u^2)\\
       &=&(1-a_1a_2u^2 )x_1^{n_1-k_1}x_2^{n_2-k_1}x_3^{n_3-k_2}x_4^{n_4-k_2}\\
       & &\prod_{i=1}^{k_1}(x_1x_2-\lambda _i^2u^2)\prod_{j=1}^{k_2}(x_3x_4-\mu _j^2u^2)\\
       &=&(1-a_1a_2u^2 )(x_1x_2-q_1q_2u^2)(x_3x_4-q_3q_4u^2)\\
       & &x_1^{n_1-k_1}x_2^{n_2-k_1}x_3^{n_3-k_2}x_4^{n_4-k_2}\\
       & &\prod_{i=2}^{k_1}(x_1x_2-\lambda _i^2u^2)\prod_{j=2}^{k_2}(x_3x_4-\mu _j^2u^2)\\
       &=&x_1^{n_1-k_1}x_2^{n_2-k_1}x_3^{n_3-k_2}x_4^{n_4-k_2}h(u)\\
       & &\prod_{i=2}^{k_1}(x_1x_2-\lambda _i^2u^2)\prod_{j=2}^{k_2}(x_3x_4-\mu _j^2u^2).
\end{eqnarray*}
Note that $|E(G)|-|V(G)|=\epsilon_1+\epsilon_2+\nu_1\nu_2-\nu_1-\nu_2$. Then by $Z_G(u)^{-1}=(1-u^2)^{|E(G)|-|V(G)|}f_G(u)$, the theorem is established.
\end{proof}

\section{The complexity}

In this section, by using the relation between $\tau(G)$ and $f_G'(1)$, We derive an expression of the complexity of the join of two semi-regular bipartite graphs in terms of their spectra.

\begin{thm}\label{cor3.2.1}\rm
Let $G=G_1\vee G_2$. Then
\begin{eqnarray*}
\tau(G)&=&(q_1+q_2+\nu_2)(q_3+q_4+\nu_1)\\
       & &(q_1+\nu_2)^{n_1-k_1}(q_2+\nu_2)^{n_2-k_1}(q_3+\nu_1)^{n_3-k_2}(q_4+\nu_1)^{n_4-k_2}\\
       & &\prod_{i=2}^{k_1}[(q_1+\nu_2)(q_2+\nu_2)-\lambda _i^2]\prod_{j=2}^{k_2}[(q_3+\nu_1)(q_4+\nu_1)-\mu _j^2].
\end{eqnarray*}
\end{thm}
\begin{proof}
Let\begin{eqnarray*}
f_1(u)&=&\prod_{i=1}^{k_1}(x_1x_2-\lambda _i^2u^2),\\
f_2(u)&=&\prod_{j=1}^{k_2}(x_3x_4-\mu _j^2u^2),\\
  g(u)&=&(1-a_1a_2u^2 )x_1^{n_1-k_1}x_2^{n_2-k_1}x_3^{n_3-k_2}x_4^{n_4-k_2}.
\end{eqnarray*}
Then $f_G(u)=f_1(u)f_2(u)g(u)$ and
\begin{alignat*}{2}
x_1|_{u=1}&=q_1+\nu_2, &\quad x_2|_{u=1}&=q_2+\nu_2,\\
x_3|_{u=1}&=q_3+\nu_1, &\quad x_4|_{u=1}&=q_4+\nu_1,
\end{alignat*}
\begin{eqnarray*}
a_1|_{u=1}=\frac {n_3(q_4+\nu_1+q_3)+n_4(q_3+\nu_1+q_4)}{(q_3+\nu_1)(q_4+\nu_1)-q_3q_4}=\frac {\nu_2(q_3+q_4+\nu_1)}{\nu_1(q_3+q_4+\nu_1)}=\frac{\nu_2}{\nu_1},\\
a_2|_{u=1}=\frac {n_1(q_2+\nu_2+q_1)+n_2(q_1+\nu_2+q_2)}{(q_1+\nu_2)(q_2+\nu_2)-q_1q_2}=\frac {\nu_1(q_1+q_2+\nu_2)}{\nu_2(q_1+q_2+\nu_2)}=\frac{\nu_1}{\nu_2}.
\end{eqnarray*}
$$(1-a_1a_2)|_{u=1}=1-\frac{\nu_2}{\nu_1}\times \frac{\nu_1}{\nu_2}=0,$$
\begin{eqnarray*}
f_1(1)&=&\prod_{i=1}^{k_1}[(q_1+\nu_2)(q_2+\nu_2)-\lambda _i^2],\label{eq13}\\
f_2(1)&=&\prod_{j=1}^{k_2}[(q_3+\nu_1)(q_4+\nu_1)-\mu _j^2],\label{eq14}\\
  g(1)&=&0.\label{eq15}
\end{eqnarray*}
Since $f'_G(u)=f'_1(u)f_2(u)g(u)+f_1(u)f'_2(u)g(u)+f_1(u)f_2(u)g'(u)$, it follows that $f'_G(1)=f_1(1)f_2(1)g'(1)$. Since
\begin{eqnarray*}
g'(u)&=&(1-a_1a_2u^2)'x_1^{n_1-k_1}x_2^{n_2-k_1}x_3^{n_3-k_2}x_4^{n_4-k_2}\\
     & &+(1-a_1a_2u^2)(n_1-k_1)x_1^{n_1-k_1-1}x_{1}'x_2^{n_2-k_1}x_3^{n_3-k_2}x_4^{n_4-k_2}\\
     & &+(1-a_1a_2u^2)(n_2-k_1)x_1^{n_1-k_1}x_2^{n_2-k_1-1}x_{2}'x_3^{n_3-k_2}x_4^{n_4-k_2}\\
     & &+(1-a_1a_2u^2)(n_3-k_2)x_1^{n_1-k_1}x_2^{n_2-k_1}x_3^{n_3-k_2-1}x_{3}'x_4^{n_4-k_2}\\
     & &+(1-a_1a_2u^2)(n_4-k_2)x_1^{n_1-k_1}x_2^{n_2-k_1}x_3^{n_3-k_2}x_4^{n_4-k_2-1}x_{4}'.
\end{eqnarray*}
and $(1-a_1a_2u^2)|_{u=1}=0$, we have
\begin{eqnarray*}
g'(1)=(1-a_1a_2u^2)'|_{u=1}(q_1+\nu_2)^{n_1-k_1}(q_2+\nu_2)^{n_2-k_1}(q_3+\nu_1)^{n_3-k_2}(q_4+\nu_1)^{n_4-k_2}.
\end{eqnarray*}
Now we compute $(1-a_1a_2u^2)'|_{u=1}$.
\begin{eqnarray*}
(1-a_1a_2u^2)'|_{u=1}&=&-(a_{1}'a_2u^2+a_1a_{2}'u^2+2a_1a_2u)|_{u=1}\nonumber\\
                        &=&-(\frac{\nu_1}{\nu_2}a_{1}'+\frac{\nu_2}{\nu_1}a_{2}'+2)|_{u=1}.\label{eq17}
\end{eqnarray*}
Since
\begin{alignat*}{2}
x_{1}^{'}|_{u=1}&=2(q_1+\nu_2-1)u|_{u=1}=2(q_1+\nu_2-1), \\
x_{2}^{'}|_{u=1}&=2(q_2+\nu_2-1)u|_{u=1}=2(q_2+\nu_2-1),\\
x_{3}^{'}|_{u=1}&=2(q_3+\nu_1-1)u|_{u=1}=2(q_3+\nu_1-1),\\
x_{4}^{'}|_{u=1}&=2(q_4+\nu_1-1)u|_{u=1}=2(q_4+\nu_1-1),
\end{alignat*}
and
\begin{eqnarray*}
a_{1}^{'}|_{u=1}&=& \left\{\frac {[n_3(x_{4}^{'}+q_3)+n_4(x_{3}^{'}+q_4)](x_3x_4-q_3q_4u^2)}{(x_3x_4-q_3q_4u^2)^2}\right.\\
         & &\left.-\frac{[n_3(x_{4}+uq_3)+n_4(x_{3}+uq_4)](x_{3}^{'}x_4+x_3x_{4}^{'}-2q_3q_4u)}{(x_3x_4-q_3q_4u^2)^2}\right\}|_{u=1},\\
         &=&\frac {\nu_1[2(q_3+q_4)\nu_2+2\nu_2(\nu_1-1)-2\epsilon_2]}{(q_3+q_4+\nu_1)\nu_{1}^{2}}\\
         & &-\frac{\nu_2[4\nu_1(q_3+q_4)+4\nu_1(\nu_1-1)+2q_3q_4-2(q_3+q_4)]}{(q_3+q_4+\nu_1)\nu_{1}^{2}},\end{eqnarray*}
and
\begin{eqnarray*}
a_{2}^{'}|_{u=1}&=&\left\{\frac{[n_1(x_{2}^{'}+q_1)+n_2(x_{1}^{'}+q_2)](x_1x_2-q_1q_2u^2)}{(x_1x_2-q_1q_2u^2)^2}\right.\\
         & &\left.-\frac{[n_1(x_{2}+uq_1)+n_2(x_{1}+uq_2)](x_{1}^{'}x_2+x_1x_{2}^{'}-2q_1q_2u)}{(x_1x_2-q_1q_2u^2)^2}\right\}|_{u=1},\\
         &=&\frac {\nu_2[2(q_1+q_2)\nu_1+2\nu_1(\nu_2-1)-2\epsilon_1]}{(q_1+q_2+\nu_2)\nu_{2}^{2}}\\
         & &-\frac{\nu_1[4\nu_2(q_1+q_2)+4\nu_2(\nu_2-1)+2q_1q_2-2(q_1+q_2)]}{(q_1+q_2+\nu_2)\nu_{2}^{2}}.
         \end{eqnarray*}
Since
$$q_1q_2\nu_1=q_1q_2n_1+q_1q_2n_2=(q_1+q_2)\epsilon_1,$$
$$q_3q_4\nu_2=q_3q_4n_3+q_3q_4n_4=(q_3+q_4)\epsilon_2,$$
we have
\begin{eqnarray*}
& &-(\frac{\nu_1}{\nu_2}a_{1}^{'}+\frac{\nu_2}{\nu_1}a_{2}^{'}+2)|_{u=1}\\
&=&-\left\{2+\frac{\nu_1[2(q_3+q_4)\nu_2+2\nu_2(\nu_1-1)-2\epsilon_2]}{(q_3+q_4+\nu_1)\nu_{1}\nu_{2}}\right.\\
     & &-\frac{\nu_2[4\nu_1(q_3+q_4)+4\nu_1(\nu_1-1)+2q_3q_4-2(q_3+q_4)]}{(q_3+q_4+\nu_1)\nu_{1}\nu_{2}}\\
     & &+\frac {\nu_2[2(q_1+q_2)\nu_1+2\nu_1(\nu_2-1)-2\epsilon_1]}{(q_1+q_2+\nu_2)\nu_{1}\nu_{2}}\\
     & &\left.-\frac{\nu_1[4\nu_2(q_1+q_2)+4\nu_2(\nu_2-1)+2q_1q_2-2(q_1+q_2)]}{(q_1+q_2+\nu_2)\nu_{1}\nu_{2}}\right\}\\
     &=&-\frac{1}{\nu_{1}\nu_{2}(q_1+q_2+\nu_2)(q_3+q_4+\nu_1)}\times \{2\nu_{1}\nu_{2}(q_1+q_2+\nu_2)(q_3+q_4+\nu_1)\\
     & &+\nu_1(q_1+q_2+\nu_2)[2(q_3+q_4)\nu_2+2\nu_2(\nu_1-1)-2\epsilon_2]\\
     & &-\nu_2(q_1+q_2+\nu_2)[4\nu_1(q_3+q_4)+4\nu_1(\nu_1-1)+2q_3q_4-2(q_3+q_4)]\\
     & &+\nu_2(q_3+q_4+\nu_1)[2(q_1+q_2)\nu_1+2\nu_1(\nu_2-1)-2\epsilon_1]\\
     & &-\nu_1(q_3+q_4+\nu_1)[4\nu_2(q_1+q_2)+4\nu_2(\nu_2-1)+2q_1q_2-2(q_1+q_2)]\}\\
     &=&-\frac{1}{\nu_{1}\nu_{2}(q_1+q_2+\nu_2)(q_3+q_4+\nu_1)}\times\{2\nu_{1}\nu_{2}(q_1+q_2+\nu_2)(q_3+q_4+\nu_1)\\
     & &+(q_1+q_2+\nu_2)[2(q_3+q_4)(\nu_2-\nu_{1}\nu_{2}-\epsilon_2)+2\nu_{1}(\nu_{2}-\nu_1\nu_2-\epsilon_2)]\\
     & &+(q_3+q_4+\nu_1)[2(q_1+q_2)(\nu_1-\nu_{1}\nu_{2}-\epsilon_1)+2\nu_{2}(\nu_{1}-\nu_1\nu_2-\epsilon_1)]\}\\
     &=&-\frac{1}{\nu_{1}\nu_{2}(q_1+q_2+\nu_2)(q_3+q_4+\nu_1)}\times\{2\nu_{1}\nu_{2}(q_1+q_2+\nu_2)(q_3+q_4+\nu_1)\\
     & &+2(q_1+q_2+\nu_2)(q_3+q_4+\nu_1)(\nu_2-\nu_{1}\nu_{2}-\epsilon_2)\\
     & &+2(q_3+q_4+\nu_1)(q_1+q_2+\nu_2)(\nu_1-\nu_{1}\nu_{2}-\epsilon_1)\}\\
     &=&\frac{2(q_1+q_2+\nu_2)(q_3+q_4+\nu_1)(-\nu_1+\nu_{1}\nu_{2}+\epsilon_1-\nu_2+\nu_{1}\nu_{2}+\epsilon_2-\nu_{1}\nu_{2})}{\nu_{1}\nu_{2}(q_1+q_2+\nu_2)(q_3+q_4+\nu_1)}\\
     &=&\frac{2(\epsilon_1+\epsilon_2+\nu_{1}\nu_{2}-\nu_1-\nu_2)}{\nu_{1}\nu_{2}}.
\end{eqnarray*}
Thus
\begin{eqnarray*}
g'(1)&=&\frac{2(\epsilon_1+\epsilon_2+\nu_1\nu_2-\nu_1-\nu_2)}{\nu_1\nu_2}\\
     & &(q_1+\nu_2)^{n_1-k_1}(q_2+\nu_2)^{n_2-k_1}(q_3+\nu_1)^{n_3-k_2}(q_4+\nu_1)^{n_4-k_2}.
\end{eqnarray*}
Since $f'_G(1)=2(m-n)\tau(G)$, it follows that
\begin{eqnarray*}
\tau(G)&=&(q_1+\nu_2)^{n_1-k_1}(q_2+\nu_2)^{n_2-k_1}(q_3+\nu_1)^{n_3-k_2}(q_4+\nu_1)^{n_4-k_2}\\
       & &\prod_{i=1}^{k_1}[(q_1+\nu_2)(q_2+\nu_2)-\lambda _i^2]\prod_{j=1}^{k_2}[(q_3+\nu_1)(q_4+\nu_1)-\mu _j^2]\frac{1}{\nu_1\nu_2}\\
       &=&(q_1+\nu_2)^{n_1-k_1}(q_2+\nu_2)^{n_2-k_1}(q_3+\nu_1)^{n_3-k_2}(q_4+\nu_1)^{n_4-k_2}\\
       & &\prod_{i=2}^{k_1}[(q_1+\nu_2)(q_2+\nu_2)-\lambda _i^2]\prod_{j=2}^{k_2}[(q_3+\nu_1)(q_4+\nu_1)-\mu _j^2]\\
       & &\frac{[(q_1+\nu_2)(q_2+\nu_2)-q_1q_2][(q_3+\nu_1)(q_4+\nu_1)-q_3q_4]}{\nu_1\nu_2}\\
       &=&(q_1+q_2+\nu_2)(q_3+q_4+\nu_1)\\
       & &(q_1+\nu_2)^{n_1-k_1}(q_2+\nu_2)^{n_2-k_1}(q_3+\nu_1)^{n_3-k_2}(q_4+\nu_1)^{n_4-k_2}\\
       & &\prod_{i=2}^{k_1}[(q_1+\nu_2)(q_2+\nu_2)-\lambda _i^2]\prod_{j=2}^{k_2}[(q_3+\nu_1)(q_4+\nu_1)-\mu _j^2].
\end{eqnarray*}
\end{proof}

\begin{rem}\rm
It is well known that $\Phi_{L}'(G, 0)=(-1)^{|V(G)|-1}|V(G)|\tau(G)$ \cite{AB}. In the case that $G=G_1\vee G_2$, $\tau(G)$ can be obtained from the direct computation of the Laplacian characteristic polynomial of $G$ by using Lemmas  \ref{lem3.1.3}, \ref{lem3.1.1} and \ref{lem3.1.2}, we leave the details to the reader.
\end{rem}

In cases the spectra of two semi-regular bipartite graphs are known, Theorem \ref{cor3.2.1} can be very useful.

\begin{cor}\rm
Let $G=K_{m,n}\vee K_{p,q}$. Then
\begin{eqnarray*}
\tau(G)&=&(m+n+p+q)^2(m+p+q)^{n-1}(n+p+q)^{m-1}\\
       & &(p+m+n)^{q-1}(q+m+n)^{p-1}.
\end{eqnarray*}
\end{cor}
\begin{proof}
Obviously, $q_1=n,q_2=m,q_3=q,q_4=p,n_1=m,n_2=n,n_3=p, n_4=q, \nu_1=m+n$ and $\nu_2=p+q$.
From \cite{AB}, we know that
\begin{eqnarray*}
spec(K_{m,n})&=&\{\pm \sqrt{mn},0,\cdot\cdot\cdot,0\},\\
spec(K_{p,q})&=&\{\pm \sqrt{pq},0,\cdot\cdot\cdot,0\}.
\end{eqnarray*}
Then the proof is straightforward  by Theorem \ref{cor3.2.1}.
\end{proof}

\begin{rem}\rm
The join $K_{m,n}\vee K_{p,q}$ is actually the complete multipartite graph $K_{m,n,p,q}$ and the complexity of the general complete multipartite graph $K_{a_1,a_2,\cdots,a_s}$ is known.
\end{rem}

\begin{example}
Let $G=C_{2m}\vee C_{2n}$. From \cite{AB}, we know that if $n$ is odd, then
$$spec(C_{2n})\!=\!\{\pm 2,\pm 2\cos\frac{\pi}{n},\pm 2\cos\frac{\pi}{n},\cdot\cdot\cdot,\pm 2\cos\frac{(n-1)\pi}{2n},\pm 2\cos\frac{(n-1)\pi}{2n}\},$$
if $n$ is even, then $$spec(C_{2n})=\{\pm 2,\pm 2\cos\frac{\pi}{n},\pm 2\cos\frac{\pi}{n},\cdot\cdot\cdot,\pm 2\cos\frac{(n-2)\pi}{2n},\pm 2\cos\frac{(n-2)\pi}{2n},0,0\}.$$
By Theorem \ref{cor3.2.1}, we show that \[\tau(G)=\left\{\begin{array}{ll}
(4+2m)(4+2n)\\
\prod_{i=1}^{\frac{m-1}{2}}[(2+2n)^2-4\cos^2\frac{i\pi}{m}]^2 & \text{both $m$ and $n$ are odd};\\
\prod_{j=1}^{\frac{n-1}{2}}[(2+2m)^2-4\cos^2\frac{i\pi}{n}]^2,\\
~~~~~\\
(4+2m)(4+2n)(2+2m)^2(2+2n)^2\\
\prod_{i=1}^{\frac{m-2}{2}}[(2+2n)^2-4\cos^2\frac{i\pi}{m}]^2 & \text{both $m$ and $n$ are even};\\
\prod_{j=1}^{\frac{n-2}{2}}[(2+2m)^2-4\cos^2\frac{i\pi}{n}]^2,\\
~~~~~\\
(4+2m)(4+2n)(2+2m)^2\\
\prod_{i=1}^{\frac{m-1}{2}}[(2+2n)^2-4\cos^2\frac{i\pi}{m}]^2 & \text{$m$ is odd and $n$ is even.}\\
\prod_{j=1}^{\frac{n-2}{2}}[(2+2m)^2-4\cos^2\frac{i\pi}{n}]^2,
\end{array}\right.\]

\end{example}

\section{The zeta function and the spectrum}\label{s5}

In this section, we explore whether the zeta function and the spectrum of the join of two semi-regular bipartite graph can determine each other.

\begin{thm}\rm
Let $G=G_1\vee G_2$, $G^{'}=G_1\vee G_{2}^{'}$. Then $Z_G(u)=Z_{G'}(u)$ if and only if $spec(G)=spec(G')$.
\end{thm}
\begin{proof}
Suppose that
\begin{eqnarray*}
spec(G)&=&\{\theta_1,\theta_2,\theta_3,\theta_4,\pm\lambda_2,\cdots,\pm\lambda_{k_{1}},\pm\mu_2,\cdots,\pm\mu_{k_{2}},\overbrace{0,\cdot\cdot\cdot\cdot\cdot\cdot,0}^{\nu_1-2k_1+\nu_2-2k_2}\},\\
spec(G')&=&\{\theta_{1}',\theta_{2}',\theta_{3}',\theta_{4}',\pm\lambda_2,\cdots,\pm\lambda_{k_{1}},\pm\mu_{2}',\cdots,\pm\mu_{k_{2}'}',\overbrace{0,\cdot\cdot\cdot\cdot\cdot\cdot,0}^{\nu_1-2k_1+\nu_{2}'-2k_{2}'}\},
\end{eqnarray*}
where $\theta_1,\theta_2,\theta_3,\theta_4$ are zeros of $f(\lambda)=\lambda^4-(q_1q_2+q_3q_4+\nu_1\nu_2)\lambda^2-2(\nu_1\epsilon_2+\nu_2\epsilon_1)\lambda+q_1q_2q_3q_4-4\epsilon_1\epsilon_2$ and
$\theta_{1}',\theta_{2}',\theta_{3}',\theta_{4}'$ are zeros of $\widetilde{f}(\lambda)=\lambda^4-(q_1q_2+q'_3q'_4+\nu_1\nu'_2)\lambda^2-2(\nu_1\epsilon'_2+\nu'_2\epsilon_1)\lambda+q_1q_2q'_3q'_4-4\epsilon_1\epsilon'_2$.

By Theorem \ref{thm3.2.1},
\begin{eqnarray*}
Z_G(u)^{-1}&=&(1-u^2)^{\epsilon_1+\epsilon_2+\nu_1\nu_2-\nu_1-\nu_2}x_1^{n_1-k_1}x_2^{n_2-k_1}x_3^{n_3-k_2}x_4^{n_4-k_2} \\
           & &\prod_{i=2}^{k_1}(x_1x_2-\lambda _i^2u^2) \prod_{j=2}^{k_2}(x_3x_4-\mu _j^2u^2)h(u) \\
 Z_{G'}(u)^{-1}
           &=&(1-u^2)^{\epsilon_1+\epsilon_{2}'+\nu_1\nu_{2}'-\nu_1-\nu_{2}'}{\widetilde{x}_1}^{n_1-k_1}{\widetilde{x}_2}^{n_2-k_1}{\widetilde{x}_3}^{n_{3}'-k_{2}'}{\widetilde{x}_4}^{n_{4}'-k_{2}'}\\
           & &\prod_{i=2}^{k_{1}}(\widetilde{x}_1\widetilde{x}_2-\lambda _{i}^{2}u^2) \prod_{j=2}^{k_{2}'}(\widetilde{x}_3\widetilde{x}_4-{\mu_{j}'}^2u^2)\widetilde{h}(u),
\end{eqnarray*}
where \begin{eqnarray*}
h(u)&=&(1-a_1a_2u^2)(x_1x_2-q_1q_2u^2)(x_3x_4-q_3q_4u^2),\\ \widetilde{h}(u)&=&(1-\widetilde{a}_1\widetilde{a}_2u^2)(\widetilde{x}_1\widetilde{x}_2-q_1q_2u^2)(\widetilde{x}_3\widetilde{x}_4-q_{3}'q_{4}'u^2),
\end{eqnarray*}
\begin{alignat*}{2}
x_1&=1+(q_1+\nu_2-1)u^2, &\quad x_2&=1+(q_2+\nu_2-1)u^2,\\
\widetilde{x}_1&=1+(q_1+\nu_{2}'-1)u^2, &\quad \widetilde{x}_2&=1+(q_2+\nu_{2}'-1)u^2,\\
x_3&=1+(q_3+\nu_1-1)u^2, &\quad x_4&=1+(q_4+\nu_1-1)u^2,\\
\widetilde{x}_3&=1+(q_{3}'+\nu_1-1)u^2, &\quad \widetilde{x}_4&=1+(q_{4}'+\nu_1-1)u^2,\\
a_1&=\frac {n_3(x_4+uq_3)+n_4(x_3+uq_4)}{x_3x_4-q_3q_4u^2}, &\quad a_2&=\frac {n_1(x_2+uq_1)+n_2(x_1+uq_2)}{x_1x_2-q_1q_2u^2},\\
\widetilde{a}_1&=\frac {n_{3}'(\widetilde{x}_4+uq_{3}')+n_{4}'(\widetilde{x}_3+uq_{4}')}{\widetilde{x}_3\widetilde{x}_4-q_{3}'q_{4}'u^2}, &\quad \widetilde{a}_2&=\frac {n_1(\widetilde{x}_2+uq_1)+n_2(\widetilde{x}_1+uq_2)}{\widetilde{x}_1\widetilde{x}_2-q_1q_2u^2}.
\end{alignat*}

($\Leftarrow$) Since $spec(G)=spec(G')$, it follows that $G$ and $G'$ have the same number of vertices, that is, $\nu_1+\nu_2=\nu_1+\nu_{2}'$ and hence $\nu_2=\nu_{2}'$. Then $x_1=\widetilde{x}_1$ and $x_2=\widetilde{x}_2$.
By Remark \ref{re1}, among $\theta_1,\theta_2,\theta_3,\theta_4$ there are no symmetric non-zero pairs. Thus $k_2=k_{2}^{'}$, $\{\pm\mu_2,\cdots,\pm\mu_{k_{2}}\}=\{\pm\mu_{2}',\cdots,\pm\mu_{k_{2}'}'\}$ and
$\{\theta_1,\theta_2,\theta_3,\theta_4\}=\{\theta_{1}',\theta_{2}',\theta_{3}',\theta_{4}'\}$.
Therefore $q_1q_2+q_3q_4+\nu_1\nu_2=q_1q_2+q_{3}'q_{4}'+\nu_{1}\nu_{2}'$,
$\nu_1\epsilon_2+\nu_2\epsilon_1=\nu_1\epsilon_{2}'+\nu_{2}'\epsilon_1$ and $q_1q_2q_3q_4-4\epsilon_1\epsilon_2=q_1q_2q_{3}'q_{4}'-4\epsilon_1\epsilon_{2}'$.
Then $q_3q_4=q_{3}'q_{4}'$ and $\epsilon_2=\epsilon_{2}'$.  In \cite{LDQ}, the authors proved that if $\nu_2=\nu_2'$, $\epsilon_2=\epsilon_2'$ and $q_3q_4=q_{3}'q_{4}'$, then $q_{3}=q_3'$ or $q_{3}=q_4'$. There are two cases.

{\bf Case 1.} If $q_3=q_{3}'$ and $q_4=q_{4}'$, then $n_3=n_{3}',n_4=n_{4}', x_3=\widetilde{x}_3$, $x_4=\widetilde{x}_4$, $a_1=\widetilde{a}_1$, $a_2=\widetilde{a}_2$. Thus $h(u)=\widetilde{h}(u)$,
\begin{alignat*}{2}
x_1^{n_1-k_1}&={\widetilde{x}_1}^{n_1-k_1},&\quad x_2^{n_2-k_1}&={\widetilde{x}_2}^{n_2-k_1},\quad \prod_{i=2}^{k_1}(x_1x_2-\lambda _i^2u^2)=\prod_{i=2}^{k_{1}}(\widetilde{x}_1\widetilde{x}_2-{\lambda _{i}}^2u^2),\\
x_3^{n_3-k_2}&={\widetilde{x}_3}^{n_{3}'-k_{2}'}, &\quad x_4^{n_4-k_2}&={\widetilde{x}_4}^{n_{4}'-k_{2}'},\quad \prod_{j=2}^{k_2}(x_3x_4-\mu _j^2u^2)=\prod_{j=2}^{k_{2}'}(\widetilde{x}_3\widetilde{x}_4-{\mu_{j}'}^2u^2).
\end{alignat*}
Hence $Z_G(u)^{-1}=Z_{{G}^{'}}(u)^{-1}$ by Theorem \ref{thm3.2.1}.

{\bf Case 2.} If $q_3=q_{4}'$ and $q_4=q_{3}'$. Then $n_3=n_{4}'$, $n_4=n_{3}', x_3=\widetilde{x}_4, x_4=\widetilde{x}_3, a_1=\widetilde{a}_1$, $a_2=\widetilde{a}_2$. Thus $h(u)=\widetilde{h}(u)$,
\begin{alignat*}{2}
x_1^{n_1-k_1}&={\widetilde{x}_1}^{n_1-k_1},&\quad x_2^{n_2-k_1}&={\widetilde{x}_2}^{n_2-k_1},\quad \prod_{i=2}^{k_1}(x_1x_2-\lambda _i^2u^2)=\prod_{i=2}^{k_{1}}(\widetilde{x}_1\widetilde{x}_2-{\lambda _{i}}^2u^2),\\
x_3^{n_3-k_2}&={\widetilde{x}_4}^{n_{4}'-k_{2}'}, &\quad x_4^{n_4-k_2}&={\widetilde{x}_3}^{n_{3}'-k_{2}'},\quad \prod_{j=2}^{k_2}(x_3x_4-\mu _j^2u^2)=\prod_{j=2}^{k_{2}'}(\widetilde{x}_3\widetilde{x}_4-{\mu_{j}'}^2u^2).
\end{alignat*}
Hence $Z_G(u)^{-1}=Z_{G'}(u)^{-1}$ by Theorem \ref{thm3.2.1}.

($\Rightarrow$)
If $G$ and $G'$ have the same zeta function, as pointed in the Introduction, they must have the same number of vertices and edges, that is, $\nu_1+\nu_2=\nu_1+\nu_{2}'$ and $\epsilon_1+\epsilon_2+\nu_1\nu_2=\epsilon_1+\epsilon_{2}'+\nu_1\nu_{2}'.$
Then $\nu_2=\nu_{2}'$ and $\epsilon_2=\epsilon_{2}'$. Thus
$x_1=\widetilde{x}_1$ and $x_2=\widetilde{x}_2$. By $Z_G(u)^{-1}=Z_{G'}(u)^{-1}$, we have
\begin{eqnarray}\label{dd}
x_3^{n_3-k_2}x_4^{n_4-k_2}\prod_{j=2}^{k_2}(x_3x_4-\mu _j^2u^2)h(u)={\widetilde{x}_3}^{n_{3}'-k_{2}'}{\widetilde{x}_4}^{n_{4}'-k_{2}'} \prod_{j=2}^{k_{2}'}(\widetilde{x}_3\widetilde{x}_4-{\mu_{j}'}^2u^2)\widetilde{h}(u).
\end{eqnarray}
And we know that
\begin{eqnarray*}
h(u)&=&(x_1x_2-q_1q_2u^2)(x_3x_4-q_3q_4u^2)-\\
& &(n_1x_2+n_2x_1+2\epsilon_1u)(n_3x_4+n_4x_3+2\epsilon_2u)u^2\\
&=&x_1x_2x_3x_4-q_3q_4u^2x_1x_2-q_1q_2u^2x_3x_4+q_1q_2q_3q_4u^4-\\
& &(n_1x_2+n_2x_1+2\epsilon_1u)(n_3x_4+n_4x_3+2\epsilon_2u)u^2\\
    &=&(q_1+\nu_2-1)(q_2+\nu_2-1)(q_3+\nu_1-1)(q_4+\nu_1-1)u^8+\\
    & &[(q_1+q_2+2\nu_2-2)(q_3+\nu_1-1)(q_4+\nu_1-1)+\\
    & &(q_3+q_4+2\nu_1-2)(q_1+\nu_2-1)(q_2+\nu_2-1)-\\
    & &q_3q_4(q_1+\nu_2-1)(q_2+\nu_2-1)-q_1q_2(q_3+\nu_1-1)(q_4+\nu_1-1)-\\
    & &(n_1q_2+n_2q_1+\nu_1\nu_2-\nu_1)(n_4q_3+n_3q_4+\nu_1\nu_2-\nu_2)]u^6-\\
    & &2[\epsilon_2(n_2q_1+n_1q_2+\nu_1\nu_2-\nu_1)+\epsilon_1(n_4q_3+n_3q_4+\nu_1\nu_2-\nu_2)]u^5+\\
    & &\cdots.
\end{eqnarray*}
Since the coefficients of the highest order of left and right hands of \eqref{dd} are the same, we have
\begin{eqnarray}
(q_3+\nu_1-1)^{n_3}(q_4+\nu_1-1)^{n_4}=(q_3'+\nu_1-1)^{n_3'}(q_4'+\nu_1-1)^{n_4'}. \label{eq31}
\end{eqnarray}
Moreover, the coefficients of the highest odd order of left and right hands of \eqref{dd} are the same, we have
\begin{eqnarray}\label{eq32}
& &[\epsilon_2(n_2q_1+n_1q_2+\nu_1\nu_2-\nu_1)+\epsilon_1(n_4q_3+n_3q_4+\nu_1\nu_2-\nu_2)]\nonumber \\
& &(q_3+\nu_1-1)^{n_3-1}(q_4+\nu_1-1)^{n_4-1}\nonumber\\
&=&[\epsilon_2(n_2q_1+n_1q_2+\nu_1\nu_2-\nu_1)+\epsilon_1(n_4'q_3'+n_3'q_4'+\nu_1\nu_2-\nu_2)]\nonumber\\
& &(q_3'+\nu_1-1)^{n_3'-1}(q_4'+\nu_1-1)^{n_4'-1}.
\end{eqnarray}
By \eqref{eq31} and \eqref{eq32} we know that
\begin{eqnarray}
& &\frac{(q_3+\nu_1-1)(q_4+\nu_1-1)}{\epsilon_2(n_2q_1+n_1q_2+\nu_1\nu_2-\nu_1)+\epsilon_1(n_4q_3+n_3q_4+\nu_1\nu_2-\nu_2)}\nonumber \\
&=&\frac{(q_3'+\nu_1-1)(q_4'+\nu_1-1)}{\epsilon_2(n_2q_1+n_1q_2+\nu_1\nu_2-\nu_1)+\epsilon_1(n_4'q_3'+n_3'q_4'+\nu_1\nu_2-\nu_2)}.\label{eq33}
\end{eqnarray}
Note that $\nu_1q_1q_2=\epsilon_1(q_1+q_2)$, $\nu_2q_3q_4=\epsilon_2(q_3+q_4)$, $\nu_2q_3'q_4'=\epsilon_2(q_3'+q_4')$. Hence $q_1+q_2=\frac{\nu_1}{\epsilon_1}q_1q_2$, $q_3+q_4=\frac{\nu_2}{\epsilon_2}q_3q_4$ and  $q_3'+q_4'=\frac{\nu_2}{\epsilon_2}q_3'q_4'$. Since $n_4q_3+n_3q_4=\nu_2(q_3+q_4)-2\epsilon_2$, $n_4'q_3'+n_3'q_4'=\nu_2(q_3'+q_4')-2\epsilon_2$,
\eqref{eq33} becomes

\begin{eqnarray}
& &\frac{(1+\frac{\nu_2(\nu_1-1)}{\epsilon_2})q_3q_4+(\nu_1-1)^2}{\epsilon_2(n_2q_1+n_1q_2+\nu_1\nu_2-\nu_1)+\epsilon_1(\frac{\nu_{2}^2}{\epsilon_2}q_3q_4-2\epsilon_2+\nu_1\nu_2-\nu_2)}\nonumber\\
&=&\frac{(1+\frac{\nu_2(\nu_1-1)}{\epsilon_2})q_3'q_4'+(\nu_1-1)^2}{\epsilon_2(n_2q_1+n_1q_2+\nu_1\nu_2-\nu_1)+\epsilon_1(\frac{\nu_{2}^2}{\epsilon_2}q_3'q_4'-2\epsilon_2+\nu_1\nu_2-\nu_2)}.\label{eq34}
\end{eqnarray}
By (\ref{eq34}), we obtain
\begin{eqnarray}
& &q_3q_4(1+\frac{\nu_2(\nu_1-1)}{\epsilon_2})[\epsilon_2(n_2q_1+n_1q_2+\nu_1\nu_2-\nu_1)+\epsilon_1(-2\epsilon_2+\nu_1\nu_2-\nu_2)]\nonumber\\
& &+q_3'q_4'\frac{(\nu_1-1)^2\epsilon_1\nu_{2}^2}{\epsilon_2}\nonumber\\
&=&q_3'q_4'(1+\frac{\nu_2(\nu_1-1)}{\epsilon_2})[\epsilon_2(n_2q_1+n_1q_2+\nu_1\nu_2-\nu_1)+\epsilon_1(-2\epsilon_2+\nu_1\nu_2-\nu_2)]\nonumber\\
& &+q_3q_4\frac{(\nu_1-1)^2\epsilon_1\nu_{2}^2}{\epsilon_2}.
\end{eqnarray}
Thus,
\begin{eqnarray*}
& &(1+\frac{\nu_2(\nu_1-1)}{\epsilon_2})[\epsilon_2(n_2q_1+n_1q_2+\nu_1\nu_2-\nu_1)+\epsilon_1(-2\epsilon_2+\nu_1\nu_2-\nu_2)]\\
& &-\frac{(\nu_1-1)^2\epsilon_1\nu_{2}^2}{\epsilon_2}\\
&=&\epsilon_2(n_2q_1+n_1q_2-2\epsilon_1+\nu_1\nu_2-\nu_1)+\\
& &\nu_2(\nu_1-1)(n_2q_1+n_1q_2-\epsilon_1+\nu_1\nu_2-\nu_1).
\end{eqnarray*}
Note that $n_2q_1+n_1q_2-2\epsilon_1=(n_2-n_1)(q_1-q_2)\geq0$, $\nu_1\nu_2-\nu_1>0$. Thus $\epsilon_2(n_2q_1+n_1q_2-2\epsilon_1+\nu_1\nu_2-\nu_1)+\nu_2(\nu_1-1)(n_2q_1+n_1q_2-\epsilon_1+\nu_1\nu_2-\nu_1)>0$
and hence $q_3q_4=q_3'q_4'$. Then $q_3=q_{3}'$ or $q_3=q_{4}'$. In the following we only consider the case $q_3=q_{3}'$ and leave the other case to the reader.

If $q_{3}'=q_3$ and $q_{4}'=q_4$, then $n_3=n_{3}', n_4=n_{4}', x_3=\widetilde{x}_3, x_4=\widetilde{x}_4, a_1=\widetilde{a}_1$ and $a_2=\widetilde{a}_2.$
Hence, $h(u)=\widetilde{h}(u)$.
Since $\mu_j\neq0$, $x_3\nmid x_{3}x_{4}-{\mu _{j}}^2u^2, x_4\nmid x_{3}x_{4}-{\lambda _{j}}^2u^2$ and $\mu_{j}'\neq0$, $\widetilde{x}_3\nmid \widetilde{x}_3\widetilde{x}_4-{\mu_{j}'}^2u^2$, $\widetilde{x}_4\nmid \widetilde{x}_3\widetilde{x}_4-{\mu_{j}'}^2u^2$, it follows that $n_3-k_2=n_{3}'-k_{2}', n_4-k_2=n_{4}'-k_{2}'.$
Then $k_2=k_{2}^{'}$. Hence
\begin{eqnarray}
\prod_{j=2}^{k_2}(x_3x_4-\mu_j^2u^2)=\prod_{j=2}^{k_{2}}(\widetilde{x}_3\widetilde{x}_4-{\mu_{j}'}^2u^2).\label{eq27}
\end{eqnarray}
If $u_1$ is a root of $x_3x_4-\mu_2^2u^2=0$, then $u_1$ is a root of $\widetilde{x}_3\widetilde{x}_4-{\mu_{j}'}^2u^2=0$. Thus
$\mu_2^2=\frac {x_3x_4}{u_{1}^2}$, ${\mu_{j}'}^2=\frac {\widetilde{x}_3\widetilde{x}_4}{u_{1}^{2}}$, $x_3x_4=\widetilde{x}_3\widetilde{x}_4$. Hence $\mu_2=\mu_{j}'$.  For $\mu_3,\ldots\mu_{k_2}$,
repeating the previous argument leads to $\{\mu_2,\mu_3,\ldots,\mu_{k_{2}}\}=\{\mu_{2}',\mu_{3}',\ldots,\mu_{k_{2}}'\}$, completing the proof.
\end{proof}

\section*{Acknowledgements}
\noindent

This work is supported by NSFC (No. 11671336) and the Fundamental Research
Funds for the Central Universities  (No. 20720190062).

\section*{References}

\bibliographystyle{model1b-num-names}
\bibliography{<your-bib-database>}

\begin{thebibliography}{99}



\bibitem{B} H. Bass, The Ihara-Selberg zeta function of a tree lattice, \emph{Internat. J. Math.} \textbf{3} (1992) 717-797.

\bibitem{BP} P. Bayati and M. Somodi, On the Ihara zeta functions of cones over regular graphs, \emph{Graphs Combin.} \textbf{29} (2013) 1633-1646.

\bibitem{BA} A. Blanchard, E. Pakala and M. Somodi, Ihara zeta function and cospectrality of joins of regular graphs, \emph{Discrete Math.} \textbf{333} (2014) 84-93.

\bibitem{AB} A. E. Brouwer and W. H. Haemers, Spectra of Graphs, Springer, New York, 2012.

\bibitem{C} H. Y. Chen and Y. Chen, Bartholdi zeta functions of generalized join graphs, \emph{Graphs Combin.} \textbf{34} (2018) 207-222.

\bibitem{CO} Y. Cooper, Properties determined by the Ihara zeta function of a graph, \emph{Electron. J. Combin.} \textbf{16} (2009) R84.

\bibitem{CD} D. Czarneski, Zeta functions of finite graphs, PhD thesis, Louisiana State University, 2005.

\bibitem{H} K. Hashimoto, Zeta functions of finite graphs and representations of p-adic groups, \emph{Adv. Stud. Pure Math.} \textbf{15} (1989) 211-280.

\bibitem{I} Y. Ihara,  On discrete subgroups of the two by two projective linear group over p-adic fields, \emph{Math. Soc. Japan} \textbf{18} (1966) 219-235.

\bibitem{LDQ} D. Q. Li and Y. P. Hou, Ihara zeta function and spectrum of the cone over a semiregular bipartite graph, \emph{Graphs Combin.} \textbf{35} (2019) 1503-1517.

\bibitem{LLH} D. Q. Li, J. Li and Y. P. Hou,  Zeta functions of several corona-type graphs, \emph{Linear Algebra Appl.} \textbf{591} (2020) 134-153.

\bibitem{LZ} X. G.  Liu and Z. H. Zhang, Spectra of subdivision-vertex join and subdivision-edge join of two graphs,   \emph{Bull. Malays. Math. Sci. Soc.}  \textbf{42} (2019) 15-31.

\bibitem{N} S. Northshield, A note on the zeta function of a graph, \emph{J. Combin. Theory Ser. B} \textbf{74} (1998) 408-410.

\bibitem{SA} I. Sato, Zeta functions and complexities of a semiregular bipartite graph and its line graph, \emph{Discrete Math.} \textbf{307} (2007) 237-245.

\bibitem{SAT}  I. Sato, Zeta functions and complexities of middle graphs of semiregular bipartite graphs, \emph{Discrete Math.} \textbf{355} (2014) 92-99.

\bibitem{AT} A. Terras, Zeta Functions of Graphs: A Stroll through the Garden, Cambridge University Press, 2011.

\bibitem{Z}  F. Z. Zhang, The schur complement and its applications, Springer, New York, 2005.















\end{thebibliography}

\end{document}